\documentclass[12pt]{amsart}

\pdfoutput=1

\usepackage{latexsym}
\usepackage[centertags]{amsmath}
\usepackage{amsfonts}
\usepackage{amssymb}
\usepackage{amsthm}
\usepackage{newlfont}
\usepackage{graphics}
\usepackage{color}

\usepackage[font=small,labelfont=bf,margin=5mm]{caption}

\RequirePackage{xparse, graphicx, caption, picins}
\DeclareDocumentCommand \addpic{O{0.4\textwidth} m g}{\parpic[r]{%
\begin{minipage}{#1}
    \includegraphics[width=\textwidth]{#2}%
    \IfNoValueTF{#3}{}{\captionof{figure}{\footnotesize #3}}
\end{minipage}
}}

\usepackage[usenames,dvipsnames]{xcolor}
\usepackage{graphicx}

\usepackage{wrapfig}

\newtheorem{theo}{Theorem}
\newtheorem{defn}[theo]{Definition}

\newtheorem{lem} [theo]{Lemma}
\newtheorem{cor}[theo]{Corollary}
\newtheorem{prop}[theo]{Proposition}
\newtheorem{rem}[theo]{Remark}

\makeatletter \@addtoreset{equation}{section}
\@addtoreset{theo}{}\makeatother

\def\n{\mathfrak{n}}

\def\bD{\overline{D}}

\def \k {{\mathbf{k}}}

\def\oD{\overline{D}}

\def\VIB{\texttt{VIB}}
\def\HIB{\texttt{HIB}}
\def\HPath{\texttt{HPath}}
\def\VPath{\texttt{VPath}}

\def\type{\textrm{type}}

\title{Inverting the General Order Sweep Map}

\author{Ying Wang$^{1}$, Guoce Xin$^2$ and Yingrui Zhang$^{3,*}$}

\address{ $^1$ School of Mathematics and Statistics, North China University of Water Resources and Electric Power, Zhengzhou 450046, PR China }
\address{ $^{2}$School of Mathematical Sciences, Capital Normal University,
 Beijing 100048, PR China }
\address{$^3$KLMM, Academy of Mathematics and Systems Science Chinese Academy of Sciences, Beijing 100190, PR China}

\email{$^1$\texttt{wangying2019@ncwu.edu.cn}\ \  \&\small $^2$\texttt{guoce\_xin@163.com} \& $^3$\texttt{zyrzuhe@126.com} }

\date{    \today }
\thanks{This work was partially supported by NSFC(12071311).}
\thanks{*Corresponding author.}

\begin{document}

\begin{abstract}
 Building upon the foundational work of Thomas and Williams on the modular sweep map, Garsia and Xin have developed a straightforward algorithm for the inversion of the sweep map on rational $(m,n)$-Dyck paths, where $(m,n)$ represents coprime pairs of integers. Our research reveals that their innovative approach readily generalizes to encompass a broader spectrum of Dyck paths. To this end, we introduce a family of Order sweep maps applicable to general Dyck paths, which are differentiated by their respective sweep orders at level $0$. We demonstrate that each of these Order sweep maps constitutes a bijective transformation. Our findings encapsulate the sweep maps for both general Dyck paths and their incomplete counterparts as specific instances within this more extensive framework.
\end{abstract}

\maketitle

\noindent
\begin{small}
 \emph{Mathematic subject classification}: Primary 05A19; Secondary 05A99 05E05.
\end{small}

\noindent
\begin{small}
\end{small}


\section{Introduction}
The study of sweep maps has been a vibrant area of research over the past few decades, with a plethora of variations and extensions being uncovered. See \cite{sweepmap}, it has been revealed that certain classical bijections are, in fact, disguised forms of the sweep map. A central challenge within this field has been the question of invertibility of the sweep map. While bijectivity has been established in several special cases, including the Fuss case where $m = kn \pm 1$, as demonstrated in \cite{Loehr-higher-qtCatalan}, ~\cite{Gorsky-Mazin2}, and ~\cite{Fuss-case}, the problem of inverting the sweep map has proven. Notably, Xin and Zhang introduced a linear algorithm for inverting the sweep map for $\k^{\pm}$-Dyck paths, as well as for $\k$-Dyck paths in \cite{Xin-Zhang}. Despite these advances, the invertibility of the sweep map for rational Dyck paths remained an open question for more than a decade.
In a seminal work, Thomas and Williams provided a general proof of the invertibility for a class of sweep maps, which they termed the modular sweep map. This breakthrough inspired Garsia and Xin to develop a geometric construction for inverting the sweep map on $(m,n)$-rational Dyck paths, where $(m,n)$ is a co-prime pair of positive integers. Their method introduced the concept of a balanced path diagram with a strictly increasing rank sequence, a construct that proved instrumental in the inversion process. Our research indicates that the insights provided by Garsia and Xin can be seamlessly generalized to apply to general Dyck paths, with the rank sequence of the balanced path diagram now exhibiting a weakly increasing trend.

In this paper, we introduce a family of Order sweep maps applicable to general paths, differentiated by their respective sweep orders at level $0$. We establish the bijection of the Order sweep map for general Dyck paths, which encompasses the sweep maps for both complete and incomplete general Dyck paths as specific instances within this more extensive framework.

Firstly, we expand upon the notation introduced in \cite{Rational-Invert}.
A path diagram $T$ of size $N$ is defined as a pair of integer sequences $T = T(P, R)$, where $P = (b_1, b_2, \dots, b_N)$ is called the \emph{path sequence}, and $R = (r_1, r_2, \dots, r_N)$ is called the \emph{rank sequence}. Visually, $T$ is represented by a sequence $(A_1, \dots, A_N)$ of $N$ arrows positioned on the $\mathbb{Z}\times \mathbb{Z}$ lattice. Each arrow $A_i$ is a $\mathbb{Z}$-vector $(1, b_i)$ starting at the lattice point $(i, r_i)$. An arrow is classified as an \emph{up vector} and is colored red if $b_i > 0$, a \emph{down vector} and colored blue if $b_i < 0$, and a \emph{level vector} and colored purple if $b_i = 0$. An illustrative example is provided in Figure \ref{fig:TableauExa}, where $P = (2, 2, 2, 0, -1, 3, 0, -4, -4)$ and $R = (1, 4, 0, 3, 2, 4, 6, 4, 5)$.

\begin{figure}[!ht]
  $$
\vcenter{ \includegraphics[height= 1.82 in]{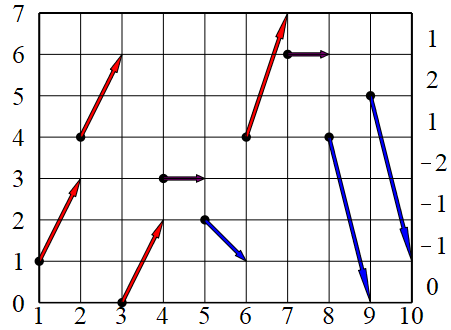}}
$$
\caption{A path diagram of size $7$.
\label{fig:TableauExa}}
\end{figure}

On the left of each horizontal lattice line, or \emph{line} for short, we have placed its $y$ coordinate which we will simply refer to as its \emph{level}.
The row of lattice cells delimited by the lines at levels $i$ and $i+1$ will be called  \emph{row $i$} or the $i$-th row. The level of the starting point of an arrow is called its \emph{starting rank}, and similarly its \emph{end rank} is the level of its end point.

Notice that each lattice cell may contain a segment of a red arrow or a segment of a blue arrow or no segment at all. Purple arrows contain no segments.
The red segment count of row $j$ will be denoted $c^r(j)$ and the blue segment count
is denoted $c^b(j)$. We will set $c(j)=c^r(j)-c^b(j)$ and refer to it as the \emph{count of row $j$}.
We also call the total segment count of a region is equal to its red segment count minus its blue segment count.
In Figure \ref{fig:TableauExa} on the right of each row we have attached its row count. It will be convenient to say that a path diagram $T(P,R)$ is \emph{balanced} if all its row counts are equal to $0$.

The following observation will be crucial in our development. Its proof is similar to \cite[Lemma $2$]{Rational-Invert}. However when we investigate the contribution to the difference $c(j) - c(j-1)$ for a
single arrow $A$, there is one more case to be considered.
\begin{lem}\label{l-DifferenceLevelCount}
Let $T(P,R)$ be any path diagram. It holds for every integer $j$
that
\begin{align}
  c(j)-c(j-1)= \# \{\text{arrows starting at level } j\} -\#\{\text{arrows end at level } j\}.
\end{align}
\end{lem}

\begin{figure}[!ht]
  $$
\vcenter{ \includegraphics[height= 1.52 in]{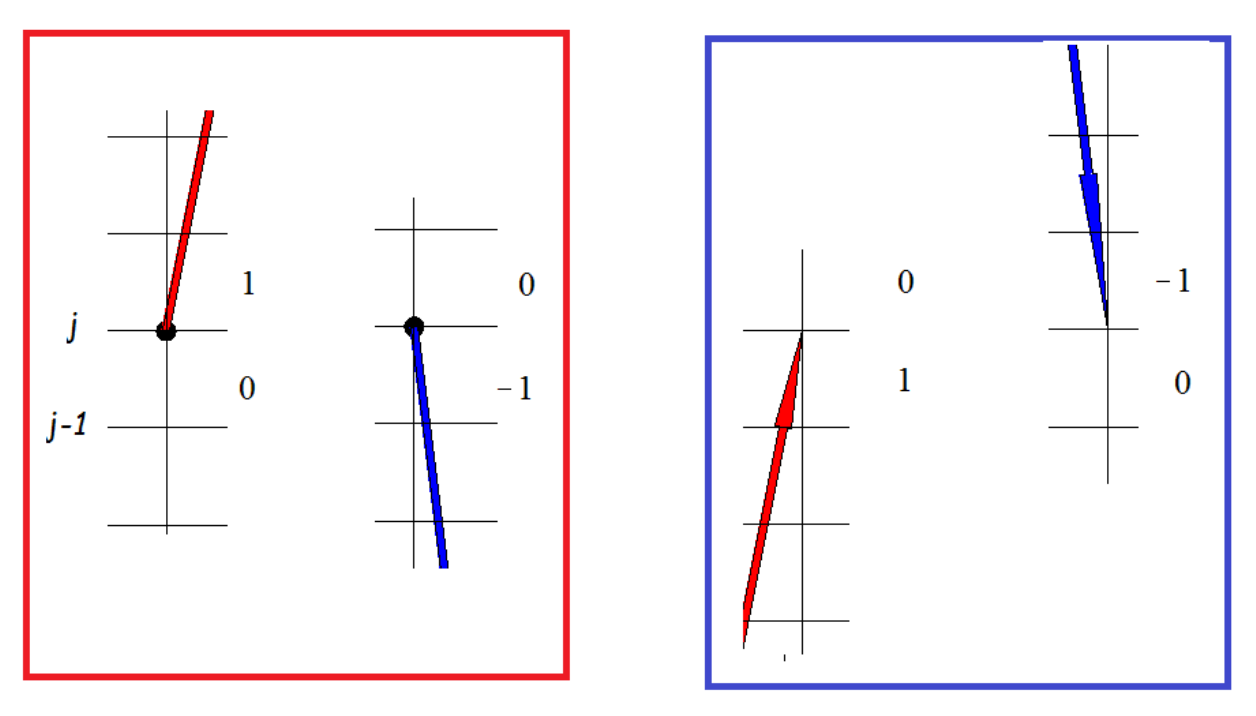}}
$$
\caption{The difference $c(j)-c(j-1)$ is $1$ in the left two cases,  is $-1$
in the right two cases.
\label{fig:FourCases}}
\end{figure}
\begin{proof}
For a single arrow $A$, the contribution to the difference $c(j)-c(j-1)$ is $0$ if i) $A$ has no segments in rows $j$ and $j-1$;
ii) $A$ has both segments in row $j$ and $j-1$;
iii) $A$ is a level vector starting at level $j$. In the first two cases,
it is clear that $A$ can not start nor end at level $j$,  and in the third case both the starting point and the ending point of A are at level $j$. The remaining cases are listed in Figure \ref{fig:FourCases}.
\end{proof}

\emph{Paths} are connected path diagrams. In what follows, paths will always end at level $0$ unless specified otherwise.
A path $P$ is usually encoded by $P=(b_1,b_2,\dots, b_N)$.
Pictorially, $P$ is a sequence of arrows $(A_1,\dots, A_N)$ where $A_i$ is the vector $(1,b_i)$, and
by connectivity, it has to starts at the point $(i, b_0+b_1+\cdots+b_{i-1})$, where $b_0$ is set to be $-b_1-\cdots-b_n$. We also write $P=A_1A_2\cdots A_N$ as concatenation of the $A_i$'s (treated as vectors).
 The \emph{type} of $P$ is the multiset $\type(P)=\{b_1,\dots, b_N\}^*$, where we use $^*$ to mean multiset. We also use the notation $s^k$ in multisets to denote $k$ copies of $s$, so
$\{1^3,2,4^2\}^*=\{1,1,1,2,4,4\}^*$.

We can think a path diagram $T=T(P,R)$ is obtained from a path $P$ by vertically shifting the arrows to start at the rank sequence $R$. Conversely, by vertical shifting of the arrows of $T$, we can uniquely obtain the path $P$, also called the V-path $P=\VPath(T)$ of $T$.
We say $T$ is (weakly) \emph{increasing} if its rank sequence $R$ is weakly increasing
and in addition if all end ranks of arrows of the path diagram $T$ are greater than or equal to $0$, we say $T$ is a \emph{positive path diagram}. Obviously, there exists the minimal increasing sequence $R$ for a positive path diagram $T$, i.e., $r_1 = \max\{0,-b_1\}$ and $r_{i+1} = \max \{r_i, -b_{i+1}\}$ for $2 \leq i \leq N-1$. We say such $T$ is a \emph{minimal path diagram}.
See the two path diagrams on the right in Figure \ref{fig:GenSweep}, a path and its minimal path diagram.

\begin{prop}\label{p-zero-row-count}
If a path $P$ starts and ends at level $0$, then it is row balanced.
\end{prop}
\begin{proof}
For a path $P$ starting at level $0$ and end at level $0$, we ignore the purple arrows. Then it is clear that
every row above level $0$ starts with a red segment and ends with a blue segment, and every row below level $0$ starts with a blue segment and ends with a red segment.
The proposition then follows since in any path, the red and blue segments alternate in each row.
\end{proof}

\begin{defn}[Sweep map]
The sweep map of a path $P$ is obtained by sorting the arrows of $P$ according
to the following two conditions:
\begin{enumerate}
  \item We first successively list arrows of starting rank $0,1,\dots,\infty $ and then list arrows of starting rank $-\infty,\dots, -2,-1.$
  \item For arrows starting at the same level, we list them from right to left.
\end{enumerate}
\end{defn}

A path $P=(b_1,b_2,\dots, b_N)$ is called a general Free path if $b_1+b_2+\cdots+b_{N}  = 0$, and
$P$ is called a general Dyck path if in addition, $b_1+\cdots+b_{i-1}\ge 0$ for all $i$.
Classical Dyck paths of length $2N$ are of type $\{1^N,(-1)^N\}^*$. For a coprime pair of positive integers $(m,n)$, classical rational $(m,n)$-Dyck paths have length $m+n$, and are of type $\{m^n,(-n)^m\}^*$. For a multiset $M=\{b_1,\dots, b_N\}^*$ with $|M|=b_1+\cdots +b_N=0,$
we denote by $\cal F_M (\cal D_M)$ the set of Free (Dyck) paths of type $M$.

\def\sweep{\texttt{sweep}}
\def\osweep{\texttt{Osweep}}
The sweep map $\sweep(\oD)$ of a Dyck path $\oD$ is defined concisely as follows: sort the arrows according to their starting ranks, if two arrows have the same starting rank, then
the right arrow comes first. Concatenation of the vectors gives the path $\sweep(\oD)$. For instance, see Figure \ref{fig:model3-Example},
the picture of $\oD=(2, 0, 2, -3, 1,-2)$ is given by the left picture, from which we easily see
that $\sweep(\oD)=(2, 1, -2, 2, 0, -3)$ is given by the right picture. Note that for the three arrows with starting rank $2$, the right arrow comes first.

\begin{figure}[!ht]
  $$
 \hskip -1.9in \vcenter{ \includegraphics[height= 1.32 in]{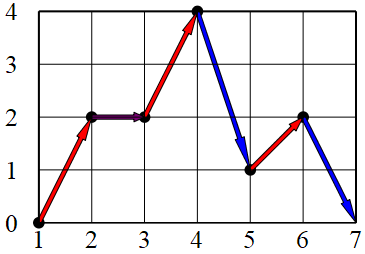}}
\hskip -2.65in  \vcenter{ \includegraphics[height=1.30 in]{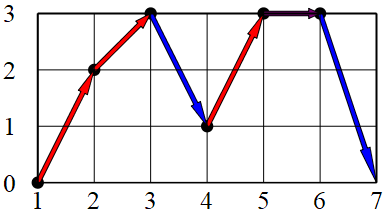}}
$$
\caption{An example of a Dyck path and its sweep map image.
\label{fig:model3-Example}}
\end{figure}

In \cite{Thomas-Williams}, H. Thomas and N. Williams repurposed the main theorem of \cite{Thomas-Williams2014} to prove that modular sweep map
is bijective and then concluded that the sweep map is bijective for general (Free) paths.
In particular, the sweep map is also a bijection for general Dyck paths.

Now let us see the second condition in definition of the sweep map. For arrows starting at the same level,
if we list arrows from left to right, then a small example will show that the sweep map is not bijective.
We aim to give a general sweep map called \emph{Order sweep} ($\osweep$ for short) and then show it is a bijection for general Dyck paths.


In order to define the Order sweep map, we use the following notations.
Let $\phi = (\phi_1, \phi_2, \dots)$ be a sequence of permutations such that $\phi_i \in \mathfrak{S}_i$. We denote by $\phi^{-1} = (\phi_1^{-1}, \phi_2^{-1}, \dots)$ the sequence of inverse permutations of $\phi$.
For a path $P = (b_1, b_2, \dots, b_N)$, let $k$ be the number of arrows starting at rank $0$.

\begin{defn}[\osweep]
The Order sweep map of a path $P$ is obtained by sorting the arrows of $P$ according
to the following two conditions:
\begin{enumerate}
  \item We first successively list arrows of starting rank $0,1,\dots,\infty $ and then list arrows of starting rank $-\infty,\dots, -2,-1.$
  \item For arrows starting at the same level, if they are starting at level $0$, we list them by the order $\phi_k$ $($from left to right$)$; otherwise, we list them from right to left.
\end{enumerate}
\end{defn}
When $\phi_k = k \ k-1 \cdots 2 \ 1$ for all $k$, it reduces to the sweep map.

For any multiset $M$ of integers with $|M|=0$ and any sequence of permutations $\phi = (\phi_1, \phi_2, \dots)$, we have the following Theorem.


\begin{theo}\label{t-mainresult}
The Order sweep map is a bijection 
$$\osweep: \cal D_M \rightarrow \cal D_M.$$
\end{theo}

The paper is organized as follows. In this introduction, we have introduced the basic concepts.
In section \ref{s-Balanced tableaux}, we give a better description of the Order sweep map and reduce the inverting Order sweep map to finding an increasing balanced path diagram.
In section \ref{s-IBT}, we describe an algorithm $\VIB$ which can output an increasing balanced path diagram.
The tightness of Algorithm $\VIB$ is given in section \ref{s-VIB} and we construct the algorithm for inverting the Order sweep map.
Finally in section \ref{s-apply}, we give an application that the sweep map and Order sweep map are also bijections for incomplete general Dyck paths.

\section{Balanced increasing path diagram}\label{s-Balanced tableaux}
Like for $(m,n)$-rational Dyck paths, increasing balanced path diagrams play an important role in inverting the Order sweep map.
Let us decompose the Order sweep map for a general Dyck path $\oD$ with $k$ arrows starting at rank $0$ as follows. See Figure \ref{fig:GenSweep} for an example.
\begin{enumerate}
  \item Use the sweep map order to label the arrows of $\oD$. That is, if we label the starting points of the arrows whose starting ranks are $0$, then we use the order $\phi_k$ to label them; otherwise,
  we label the starting points of the arrows, from bottom to top, and if two points are in the same level, from right to left.
  \item Horizontally shift the arrows so that small labels comes first. Denote this path diagram by $T=\HIB^{\phi_k}(\oD)$.
  \item Set $\osweep(\oD)=\VPath(T)$, the V-path of $T$.
\end{enumerate}

\begin{figure}[!ht]
\centering{
\mbox{\includegraphics[height=3.30 in]{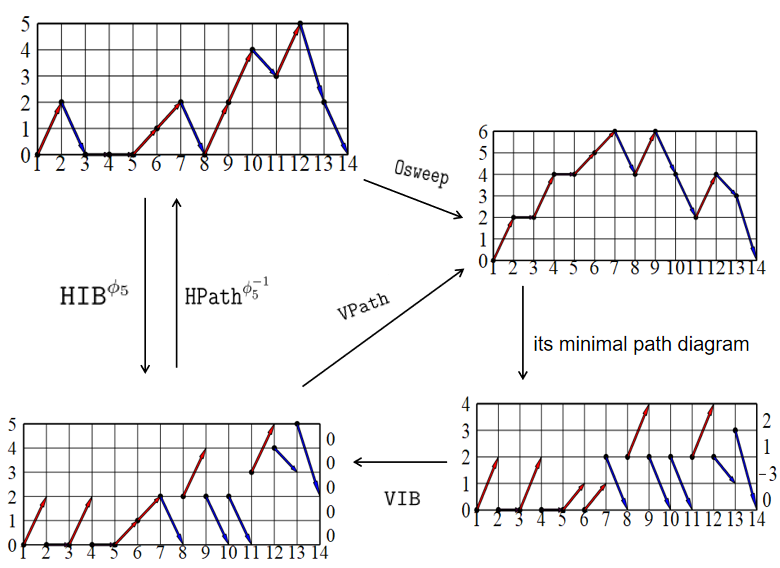}}
\caption{An example of the decomposition of the Order sweep map with $\phi_5 = 13524$ and $\phi_5^{-1} = 14253$. The notations $\HPath^{\phi_5^{-1}}$ and $\VIB$ will be defined later.}
\label{fig:GenSweep}
}
\end{figure}

The path diagram $T=\HIB^{\phi_k}(\oD)$ is clearly increasing. It is balanced since $\oD$ is balanced and horizontal shifts preserve row counts.
\begin{prop}\label{p-VPath2Dyckpath}
  The V-path of an increasing balanced path diagram is a Dyck path.
\end{prop}
\begin{proof}
Let $T$ be an increasing balanced path diagram. Pick any $i$, the picture of $T$ looks like  Figure \ref{fig:VPathDyck},
where the two green lines split the plane into four regions as displayed. We need to show that the total segment count in $A$ and $C$ is nonnegative.
\begin{figure}[!ht]
\centering{
\mbox{\includegraphics[height=2.6 in]{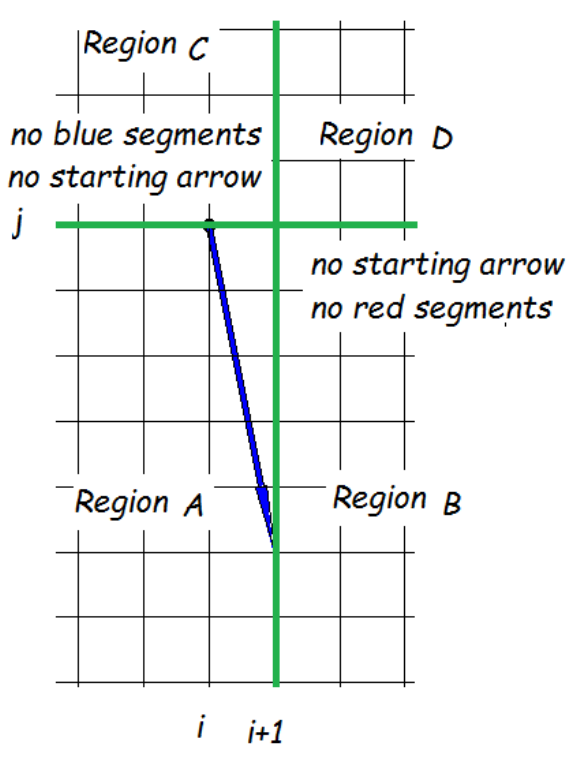}}
\caption{The picture here shows that the V-path is a Dyck path.}
\label{fig:VPathDyck}}
\end{figure}

By the increasing hypothesis of the starting ranks of $T$, region $C$ may only have red segments and region $B$ may only have blue segments.
Additionally the balance property of $T$ shows that the total segment count of region $A$ and $B$ is $0$. This forces the segment count of $A$ to be nonnegative. So the total segment count in $A$ and $C$ is nonnegative.
\end{proof}

\begin{cor}\label{c-Orderimage}
The Order sweep image of a general Dyck path is a general Dyck path.
\end{cor}


\begin{proof}
By definition, the Order sweep image of a general Dyck path $\oD$ is represented by the V-path of the path diagram $T$, denoted as $\HIB^{\phi_k}(\oD)$. Since $\HIB^{\phi_k}(\oD)$ constitutes an increasing balanced path diagram, it follows directly from Proposition \ref{p-VPath2Dyckpath}.
\end{proof}

Now, we give an algorithm $\HPath^{\phi_k^{-1}}$ for constructing Order sweep map pre-image from an increasing balanced path diagram $T=(D,R)$ with a sequence $(A_1,\dots, A_N)$ of $N$ arrows.

\noindent
Algorithm $\HPath^{\phi_k^{-1}}$:  

\textbf{Input:} An increasing balanced path diagram $T=(D,R)$ with a sequence $(A_1,\dots, A_N)$.

\textbf{Output:} A Dyck path $\bD=\HPath^{\phi_k^{-1}}(T)$.
\begin{enumerate}
\item Set $r_1 = 0$ as the current level and $\n = 0$. Set or update $k$ to be the number of arrows starting at level $0$.

\item For $i$ from $1$ to $N$ do the following.

          If $r_i = 0$ then set $\n = \n+1$. Find the $\phi_k^{-1}(\n)$-th arrow from left to right at level $0$ , say $A_j$;

          Otherwise, find the rightmost unlabeled arrow, say $A_j$, that starts at the current level.

         If no such $j$ exists then go to step 3;

         Otherwise label $A_j$ by $i$, set $\pi(i)=j$, and update the current level as the end rank of $A_{\pi(i)}$.

\item When step 2 stops early with no unlabelled arrow found, then update $T$ by shifting all unlabelled arrows one level down. Erase all labels and go to step 1.

\item If $D=(A_1,\dots, A_N)$, then output $\bD=(A_{\pi(1)},A_{\pi(2)},\dots, A_{\pi(N)})$.
\end{enumerate}

If an increasing balanced path diagram $T$ is updated by the above step $3$ eventually to $T'$ whose all of the arrows can be labelled in step $2$, then we call such $T'$ \emph{stable}. (Note that the step $3$ may not be used.)

Here we give an example to show an increasing balanced path diagram and its stable balanced path diagram.
\begin{figure}[!ht]
  $$
 \hskip -1.75in\vcenter{ \includegraphics[height= 0.96 in]{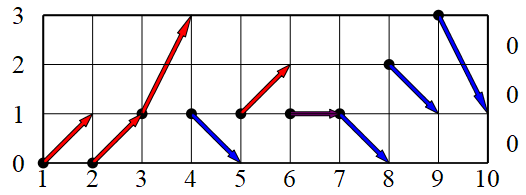}}
 \hskip -3.05in \vcenter{ \includegraphics[height=0.80 in]{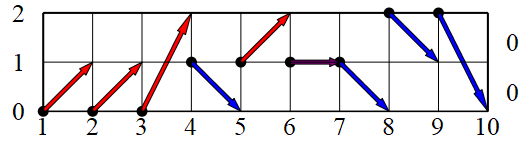}}
$$
\caption{An increasing balanced path diagram and its corresponding stable balanced path diagram.}
\end{figure}

\begin{lem}\label{l-step2balanced}
In Algorithm $\HPath^{\phi_k^{-1}}$, step 2 stops early with no unlabelled arrow found. Assume that the labeled arrows are $A_{\pi(1)}, A_{\pi(2)}, \cdots ,A_{\pi(i)}$ where $i < N$. Then we can obtain a Dyck path by horizontally shifting of the labeled arrows, i.e., $P=A_{\pi(1)}A_{\pi(2)}\cdots A_{\pi(i)}$ is a Dyck path.
It is balanced and the unlabled arrows is also balanced.
\end{lem}
\begin{proof}
In Step 2, the process terminates at $A_{\pi(i)}$ with $i < N$. Suppose that the end rank of the arrow $A_{\pi(i)}$ is $j$. We assert that $j = 0$ and that all arrows originating at rank $0$ are labeled. If this were not the case, then we would have (repeatedly) encountered $j > 0$ as an end rank, followed by $j$ as a starting rank upon departure, ultimately terminating with $j$ as the end rank of $A_{\pi(i)}$. Consequently, the count of arrows ending at level $j$ would exceed the count of arrows starting at level $j$.
However, these two counts must be equal by Lemma \ref{l-DifferenceLevelCount}, and the input diagram $T$ is balanced, or more specifically, $c(j) - c(j-1) = 0$ for all $j$. This implies that the path $P = A_{\pi(1)}A_{\pi(2)}\cdots A_{\pi(i)}$ initiates and concludes at level $0$. By Proposition \ref{p-zero-row-count}, we infer that $P$ is balanced. Furthermore, it is readily apparent that $P$ is a Dyck path, and the remaining unlabeled arrows also form a balanced diagram, since the removal of a balanced sub-diagram from a balanced diagram necessitates that the residual diagram remains balanced.
\end{proof}

\begin{lem}\label{l-updatebalanced}
In Algorithm $\HPath^{\phi_k^{-1}}$ step 3 updates $T$ to a new increasing balanced path diagram with $r_1=0$. Moreover, a stable path diagram is an increasing balanced path diagram.
\end{lem}
\begin{proof}
Consider that Step 2 concludes at $A_{\pi(i)}$ with an end rank of $j$, where $i < N$. This implies that no new arrows commence at rank $j$.
By Lemma \ref{l-step2balanced}, we deduce that $j = 0$, and the sequence of labeled arrows forms a Dyck path $P = A_{\pi(1)} \cdots A_{\pi(i)}$. This path is balanced, and the remaining unlabeled arrows also constitute a balanced diagram.
Hence, by descending all unlabelled arrows by one level, we maintain an \emph{increasing} balanced path diagram with $r_1 = 0$. The increasing property is retained because the inequality $r_a \leq r_b$ for $a < b$ can only be violated if $0 < r_a = r_b$ and $A_a$ is labeled while $A_b$ is not, which contradicts the selections made in Step 2. (It is important to note that the preceding argument relies on the assumption that the sweep order for each positive level is from right to left.)
\end{proof}

\begin{rem}
In \cite{Rational-Invert}, Garsia and Xin defined the \HPath \ algorithm for constructing  sweep map pre-image which need not use step 3.
This is because the coprime condition $(m,n)=1$ forces every increasing balanced path diagram of type $\{m^n,(-n)^m\}^*$ has to be stable.
Moreover, there is only one arrow with starting rank $0$.
\end{rem}

\begin{lem}
It is independent of $\phi_k^{-1}$ that we determine whether an increasing balanced path diagram is a stable balanced path diagram or not.
\end{lem}

\begin{proof}
Consider an increasing balanced path diagram $T$ that is not stable. By Lemma \ref{l-step2balanced}, we may assume that Step 2 terminates at $A_{\pi(i)}$ with an end rank of $j = 0$, for $i < N$. The labeled arrows are $A_{\pi(1)}, A_{\pi(2)},\dots, A_{\pi(i)}$, and there are $k$ arrows $A_{\pi(i_1)}, A_{\pi(i_2)}, \dots, A_{\pi(i_k)}$ that start at level $0$.
We interpret $T$ as a directed graph with multiple edges and potential loops. The rank sequence of $T$ (where each rank occurs only once) constitutes its vertex set, and the set of arrows in $T$ forms its edge set. It is a property of $T$ that for every vertex, the outdegree is equal to the indegree.
For clarity, let us denote level $0$ as the vertex $j_0$, and let the end ranks of $A_{\pi(i_1)}, A_{\pi(i_2)}, \dots, A_{\pi(i_k)}$ be the distinct vertices $j_1, j_2, \dots, j_{k'}$. We disregard the possibility of loops at vertex $j_0$, as this does not impact our conclusion.
By Algorithm $\HPath^{\phi_k^{-1}}$, the selection of an arrow when leaving each vertex is uniquely determined.

We construct a directed graph $T'$ by choosing the edges $A_{\pi(1)},\dots, A_{\pi(i)}$. It is easy to see that $A_{\pi(1)},\dots, A_{\pi(i)}$ features an Eulerian tour of $T'$ that commences and concludes at vertex $j_0$, encompassing all edges originating from or terminating at vertex $j_0$. Altering the order of departure from vertex $j_0$ does not affect the entry and exit of vertices $j_{p}$ for $1 \leq p \leq k'$. Consequently, the reordering always yields an Eulerian tour of the directed graph $T'$.
When $T$ is stable, $T'$ is identical to $T$. This completes the proof.
\end{proof}

The following Proposition implies that increasing balanced path diagrams play a crucial role in
inverting the Order sweep map.

\begin{prop}\label{p-Bcrucialrole}
If $T$ is an increasing balanced path diagram, then Algorithm $\HPath^{\phi_k^{-1}}$ outputs a Dyck path $\bD=\HPath^{\phi_k^{-1}}(T)$ satisfying $\osweep(\bD)=\VPath(T)$.
\end{prop}
\begin{proof}
By Lemma \ref{l-updatebalanced}, we know that step 3 updates $T$ to a new increasing balanced path diagram with $r_1=0$.
Since step 3 always shifts some arrows downwards, it can not be performed infinitely many times. Therefore, we finally reached the position when performing step 2 on the current increasing balanced path diagram $T$, all the $N$ arrows have been labelled. In other words, it becomes a stable balanced path diagram. This assertion implies that each increasing balanced $T$ will be updated to a unique stable one.
Then setting $\oD=A_{\pi(1)}\cdots A_{\pi(N)}$ clearly gives a Dyck path.  On the other hand, the
increasing balanced path diagram $\HIB^{\phi_k}(\oD)$ obtained from $\oD$ through horizontal shifts is exactly equal to $T$. The increasing property forces the relative order for arrows of different ranks. For two arrows of the same rank $> 0$, $A_{\pi(a)}$ appears to the right of $A_{\pi(b)}$ if $a<b$, and their order is reversed in $\oD$. But in $\HIB^{\phi_k}(\oD)$ their order is also reversed according to their position in $\oD$.

There is a difference for arrows with starting rank $0$. Assume there are $k$ such arrows in a stable $T$, say $B_1,\dots,B_k$ from left to right. Then
their relative order in $\oD$ is $B_{\phi_k^{-1}(1)},\dots, B_{\phi_k^{-1}(k)}$ from left to right. Now in
$\osweep(\oD)$, the relative order of these arrows becomes $B_{\phi_k(\phi_k^{-1}(1))},\dots, B_{\phi_k(\phi_k^{-1}(k))}$, which is just
 $B_1,B_2,\dots, B_k$. Thus $T$ and $\HIB^{\phi_k}(\oD)$ have to be identical and we have $\osweep(\oD)=\VPath(T)$ as desired.
\end{proof}

\section{The Increasing Balanced path diagram obtained by Vertical Shifting}\label{s-IBT}
Given a general Dyck path $D$, we have reduced to finding an increasing balanced path diagram $T$ such that $\VPath(T)=D$.
This can be achieved by the following algorithm.

\noindent
Algorithm \VIB\ :

\noindent
\textbf{Input:} A positive path diagram $T(D,R^{(0)})$, where $D$ is a Dyck path in  $\cal D_{M}$,
and $R^{(0)} =(r_1^{(0)},r_2^{(0)},\ldots ,r_{N}^{(0)})$ is any weakly increasing (rank) sequence.

\noindent
\textbf{Output:} An increasing balanced path diagram $T(D,\overline{R})$.

\begin{enumerate}
\item[Step 1]  Starting with $T(D,R^{(0)})$ repeat the following step
until the resulting path diagram is balanced.
\vskip .1in

\item[Step 2]
{ In $T(D,R^{(s)})$,  with  $R^{(s)} =(r_1^{(s)},r_2^{(s)},\ldots ,r_{N}^{(s)}),$  find the lowest row $j$ with $c(j)>0$ and
find the rightmost arrow that starts at level $j$. We say that we are \emph{working} with row $j$. Suppose that arrow
starts at $(i,j)$. Move up the arrow one level to construct the tableau  $T(D,R^{(s+1)})$ with $ r_{i}^{(s+1)}= r_{i}^{(s)}+1$
and  $ r_{i'}^{(s+1)}= r_{i'}^{(s)}$ for all $i'\neq i$.
If all the row counts are $\le 0$ then stop the algorithm, since all row counts
must  necessarily vanish.}
\end{enumerate}



For a path diagram $T(D,R^{(0)})$ as in Algorithm $\VIB$, let $U$ be the maximal value in set $S = \{r_i^{(0)}+b_i | b_i > 0\} \bigcup \{r_j^{(0)}| b_j < 0\}$. In other words, $U$ is the number of the highest level of the path diagram $T(D,R^{(0)})$ achieving.
It also makes the following key observation evident.
\begin{lem}\label{l-keep-nonnegative}
If  at some point  $c(k)$ becomes $\ge 0$ then for ever after it will never become negative. In particular, since $c(k)= 0$ with $k>U$ for the initial path diagram $T(D,R^{(0)})$ we will have $c(k)\ge 0$ when $k>U$ for all successive path diagrams produced by the algorithm.
\end{lem}
\begin{proof}
This Lemma follows that we only decrease a row count when it is positive.
\end{proof}

We need  to justify the algorithm by using some basic properties. 
\begin{lem}\label{l-basic-property}
We have the following basic properties.
\begin{enumerate}
\setlength\itemsep{2mm}
\item[$(i)$] { If row $j$ is the lowest with $c(j)>0$ then there is at least one arrow that starts at level $j$.}

\item[$(ii)$] { The successive rank sequences are  always weakly  increasing.}

\item[$(iii)$]  { If $T(D,R^{(s)})$ has no positive row counts, then it is balanced. Consequently, if the algorithm terminates, the last path diagram is balanced.}
\end{enumerate}
\end{lem}

\begin{proof}\

\begin{enumerate}\setlength\itemsep{2mm}

\item[$(i)$]    By the choice of $j$, we have $c(j)>0$ and $c(j-1)\le 0$. Thus $c(j)-c(j-1)>0$, which by Lemma \ref{l-DifferenceLevelCount}, shows that there is at least one arrow starting at rank $j$.

\item[$(ii)$]   Our choice of $i$ in Algorithm $\VIB$ step $2$ assures that the next rank
sequence remains weakly increasing.

\item[$(iii)$]  Since each of our tableaux  $T(D,R^{(s)})$ has the same number of red segments and blue segments, the total sum of row counts of any $T(D,R^{(s)})$ has to be $0$. Thus if $T(D,R^{(s)})$ has no positive row counts, then it must have no negative row counts either, and is hence balanced.
 \end{enumerate}
\vspace{-8mm}
 \end{proof}

\begin{proof}[Justification of Algorithm \VIB]
By Lemma \ref{l-basic-property}, we only need to show that the algorithm terminates.
To prove this we need the following auxiliary result.

\begin{lem}\label{l-redcount}
Suppose we are working with row $k$, that is  $c(k)>0$ and $c(i)\le 0$ for all $i<k$. If row $\ell $ has no segments for some $\ell <k$, then in the current path diagram $T(D,R')$, it holds that
\begin{align}\label{e-eq-terminate}
  c(0) = c(1) = \cdots = c(\ell-1) = 0.
\end{align}
\end{lem}
\begin{proof}\hspace{-2cm}\vspace{-1cm}\addpic[.38\textwidth]{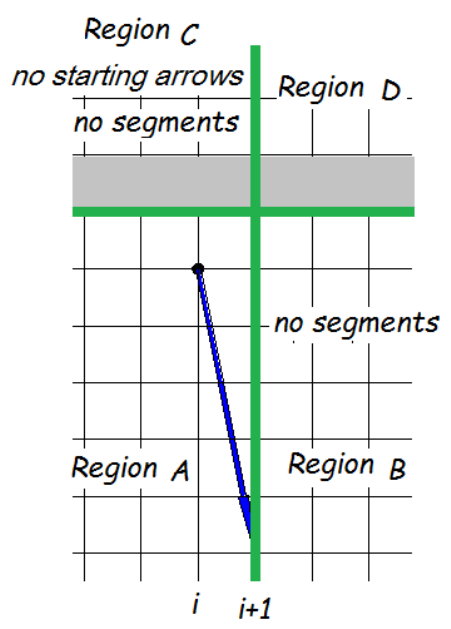}{When row $\ell$ has no segments, there will be no segments in regions $B$ and  $C$.\label{fig:emptyrow}}
\vspace{.1cm} \hspace{5mm}

The conclusion of this Lemma is different from that in \cite[Lemma 6]{Rational-Invert}, but the proof is very similar.

If there are no segments below row $\ell$ in the current path diagram $T(D,R')$, then the Lemma holds true.
Now, we suppose that $T(D,R')$ has a segment below row $\ell$, then we only need to show $\sum_{0\le j <\ell } c(j) =0.$
Let $\bar{A}$ be the rightmost arrow containing such a segment, starting at column $i$.
Since row $\ell$ has no segments, the starting rank of $\bar{A}$ must be $\le \ell$. This implies that $i+1<N$ since the arrow starting at level $k$ must be to the right of $\bar{A}$ (by the increasing property of $R'$).  This given,  the current path diagram could look like in
Figure \ref{fig:emptyrow}, where the two green lines divide the plane into 4 regions, as labelled in the display. The gray row is row $\ell$.

The weakly increasing property of $R'$  implies no starting ranks in  $C$, therefore there are no segments there. Due to the choice of $\bar{A}$, there can be no segments in $B$ either.
Thus the  (gray) empty row $\ell$  forces no segments within both $B$ and $C$.

Now observe that since $D\in \mathcal D_{M}$, the difference between the number of red segments to the left of column $i+1$ and the number of blue segments to the left of that column must be a non-negative number $s\ge 0$.
However, since $c(j)\le 0$ for all $j\le \ell$ it follows that $c(0)+\cdots +c(\ell-1)\le 0$. Additionally, as regions $B$ and $C$ contain no segments, it follows that $s=c(0)+\cdots +c(\ell-1)\le 0$. This implies $s=0$ and \eqref{e-eq-terminate} holds true.
\end{proof}

 \vskip -.1in
Next observe that since each step of the algorithm increases one of the ranks by one unit, after $K$ steps we will have  $|R^{(K)}-R^{(0)}|=\sum_{i=1}^{N}r_i^{(K)}-\sum_{i=1}^{N}r_i^{(0)} =K$. With this in mind, if the algorithm iterates Step 2 forever, then the maximum rank  will eventually exceed any given integer. In particular, we will reach a point where we are working with row $k$, and $k$ is so large that $k-U$ exceeds the total number $\sum_{i=1, b_i>0}^N b_i$ of red segments.
At this stage, we will have $c(k)>0$ and  $c(j)=0$ for all the $k-U$ values $j=U,U+1,\dots, k-1$.
The reason for this is that we must have $c(j)\le 0$ for all $0\le j < k$ and by Lemma \ref{l-keep-nonnegative} we must also have
$c(j)\ge 0$ for $j\ge U$. Now, by the  pigeon hole principle, there must also be
some $U\le \ell <k$ for which $c^r(\ell )=0$. But then it follows that $c^b(\ell)=c^r(\ell)-c(\ell)=0$, too. That implies that row  $\ell$ contains no segments. Utilizing Lemma \ref{l-redcount}, we obtain equation \eqref{e-eq-terminate}. Therefore, the total row count is given by $\sum_{j\ge 0} c(j) = \sum_{j\ge \ell} c(j) \ge c(k)>0$, a contradiction.

\end{proof}

\def\TR{\widetilde{R}}
\def\tr{\tilde{r}}
\section{The tightness of Algorithm \VIB}\label{s-VIB}
For  two rank sequences $R=(r_1,r_2,\ldots ,r_{N})$ and
$R'=(r_1',r_2',\ldots ,r_{N}')$ let us write $R\preceq R'$ if and only if
we have $r_i\le r_i'$ for all $1\le i\le N$; if $r_i< r_i'$ for at least one $i$
we will write $R\prec R'$. The \emph{distance} of $R$ from $R'$,
will be expressed by the integer
 $$
| R'-R| \, =\, \sum_{i=1}^{N} (r_i'-r_i)=\, \sum_{i=1}^{N} r_i'- \sum_{i=1}^{N} r_i.
$$

Given a general Dyck path $D$, we can obtain the minimal path diagram $T(D,R^{(0)})$ by vertically shifting the arrows. We
call the initial starting sequence $R^{(0)}$ \emph{canonical for $D$}. Let $T(D,\TR)$ be the increasing balanced path diagram obtained by Algorithm \VIB \ from $T(D,R^{(0)})$. Then, we have the following theorem.

\begin{theo}\label{t-tight}
If $R$ is any  increasing sequence which satisfies the inequalities
$$
R^{(0)} \preceq R\preceq \TR
$$
then the \VIB \ algorithm with starting path diagram $T(D,R)$
will have as output the rank sequence $\TR$.
\end{theo}

\begin{proof}
If $|\TR-R|=0$, there is nothing to prove. Therefore, we will proceed
by induction on the distance of $R$ from $\TR$.
Now assume the theorem holds for $|\TR-R|=K$.
We need to show that it also holds for $|\TR-R|=K+1$.
Suppose one application of step (2) on $R$ gives $R'$.
We aim to show that $R'\preceq \TR $. This done since $R$ and $R'$ only differ from one unit  we will have $|\TR-R'|=K$ and then the inductive hypothesis would complete the proof.

Thus assume if possible that this step (2) cannot be carried out because it requires increasing by one unit an $r_i=\tr_i$ .
Suppose further that under this  step (2) the level $k$ was the lowest with $c(k)>0$ and thus the arrow $A_i$ was the right most that started at level $k$. In particular this means that $r_i=\tr_i=k$. Since $|  \TR-R|\ge 1$, there is at least one $i'$ such that $r_{i'}<\tr_{i'}$. If $r_{i'}=k'$ let $i'$ be the right most
with $r_{i'}=k'$. Define $R''$ to be the rank sequence obtained by replacing  $r_{i'}$ by
$r_{i'}+1$ in $R$. The row count $c(k')$ is decreased by $1$ and another row count is increased by $1$, so that $c(k)$ in $R''$ is still positive.
Since $|\TR-R''|=K$ the induction hypothesis assures that
the  \VIB \ algorithm will return $\TR$. But then in carrying this out, we have to work on row $k$, sooner or later, to decrease the positive row count $c(k)$. But there is no way the arrow $A_i$ can stop being the right most starting at level $k$, since arrows to the right of $A_i$ start at a higher level than $A_i$  and  are only moving upwards.   Thus the fact that   the  \VIB \  algorithm outputs  $\TR$ contradicts the step (2) applied to $R$ cannot be carried out.
Thus we will be able to lift $A_i$ one level up as needed and obtain the sequence $R'$ obtained by replacing $r_i$ by $r_i+1$ in $R$.
But now  we will have $|\TR-R'|$ =K$ $ with  $R'\prec \TR$. By the inductive hypothesis, the \VIB\ algorithm starting from $R$ will return $\TR$ as desired, completing the proof.
\end{proof}

Now we can describe the algorithm for inverting the Order sweep map with a sequence of permutations $\phi = (\phi_1,\phi_2,\dots)$.

\noindent
Algorithm {\bf InvOSweep} :

\noindent
Input: Dyck path $D$.

\noindent
Output: Dyck path $\oD$ with $\osweep(\oD)=D$.
\begin{enumerate}
\item Obtain the minimal path diagram $T(D,R)$ by vertically shifting the arrows

\item Compute the increasing balanced path diagram $\VIB(T)$, say equal to $T(D,\bar{R})$.

\item Compute $\HPath^{\phi_k^{-1}}(D,\bar R)$. This is the desired Dyck path $\oD$.
\end{enumerate}

By the tightness of Algorithm \VIB, we can claim that Step 2 indeed gives a stable balanced path diagram so that step 3 in Algorithm $\HPath^{\phi_k^{-1}}$ is not used. Suppose that Step 2 gives a balanced path diagram $T = T(D, \bar {R})$ which is not stable. By Lemma \ref{l-updatebalanced}, we can obtain a new increasing balanced path diagram $T' = T(D, \bar R')$ by Algorithm $\HPath^{\phi_k^{-1}}$ step $3$ such that $R \preceq \bar{R}' \preceq \bar{R}$, which is a contradiction.

\begin{rem}
It might be better to guess a starting $R$ and do the computation.
\end{rem}

\begin{proof}[Proof of Theorem \ref{t-mainresult}]
By Corollary \ref{c-Orderimage}, the Order sweep image of a general Dyck path is a general Dyck path, i.e, it is into.
Since the collection $\cal D_M$ of a given multiset $M$ is finite, we only show that
 given a Dyck path $D \in \cal D_M$, one can obtain $\oD \in \cal D_M$ such that $\osweep(\oD) = D$. This can be done by Algorithm {\bf InvOSweep}.
\end{proof}


\section{An application}\label{s-apply}

A path encoded by $P^o=(b_1,b_2,\cdots,b_N)$ is called an incomplete Dyck path written by $D^o$ if it has to start at the point $(0,a)$, where
$a$ is equal to $- \sum_{i=1}^N b_i$ and greater than zero, and for all $i$, $a + b_1+ \cdots +b_i \geq 0$ and $a + b_1+ \cdots +b_N = 0$. It is clear to know that its \emph{rank sequence} $R^o=(r_1,r_2,\dots, r_N)$ satisfies the conditions $r_1= a > 0$ and $r_i \geq 0$ for $2 \leq i \leq N$.
Note that if we add a red arrow $A^o$ at the point $(-1, 0)$ where $A^0$ is the vector $(1,a)$, then it becomes a Dyck path
of length $N+1$ starting at $(-1,0)$ and ending at $(N,0)$. So, when we ignore the x-coordinate, an incomplete Dyck path can be
obtained by removing the first step for a general Dyck path of length $N+1$.

In what follows, for a multiset $B=\{b_1,\dots, b_N\}^*$ with $|B|=b_1+\cdots +b_N<0,$  we denote by $\cal D^o_B$ the set of incomplete Dyck paths of type $B$. For an incomplete Dyck path $D^o \in \cal D^o_B$ and its rank sequence $R^o=(r_1,r_2,\dots, r_N)$, if we add a red dashed arrow $(1, r_1)$ before the first step of it, then it will becomes a general Dyck path denoted by $D = \mathcal{A}(D^o)$. We write $D^o = \mathcal{A}^{-1}(D)$.

We give the following theorem.
\begin{theo}
For any multiset $B$ of integers with $|B|<0$, the sweep map $\sweep$ is a bijection 
$$\sweep : \cal D^o_B \rightarrow \cal D^o_B.$$
\end{theo}

\begin{proof}
Since $$\sweep = \mathcal{A}^{-1} \circ \osweep \circ \mathcal{A},$$
which the permutation sequence $\phi = (\phi_1, \phi_2,\dots)$ satisfies $\phi_k = 1\ k \ k-1\ \cdots 2$ for all $k$.
It follows by a consequence of Theorem \ref{t-mainresult}.
\end{proof}

Here, we present an illustrative example. 

Consider a given incomplete Dyck path $D^o$. The sweep map image of $D^o$ can be constructed through the following steps:
\begin{description}
  \item[1)] Add a red arrow to $D^o$ to produce a Dyck path, which we denote as $D$;
  \item[2)] Derive the Order sweep map image, $P(T)$, corresponding to $D$;
  \item[3)] Remove the initial step from $P(T)$, thereby yielding the desired incomplete Dyck path.
\end{description}
On the other hand, given an incomplete Dyck path $P^o(T)$, we want to obtain the preimage of $P^o(T)$ under the sweep map. 
It can be constructed through the following steps:
\begin{description}
  \item[1)] Add a red arrow to $P^o(T)$ to produce a Dyck path, which we denote as $P(T)$;
  \item[2)] Derive the Order sweep map preimage, $D$, corresponding to $P(T)$;
  \item[3)] Remove the initial step from $D$, thereby yielding the desired incomplete Dyck path.
\end{description}

\begin{figure}[!ht]
\centering{
\mbox{\includegraphics[height=3.2 in]{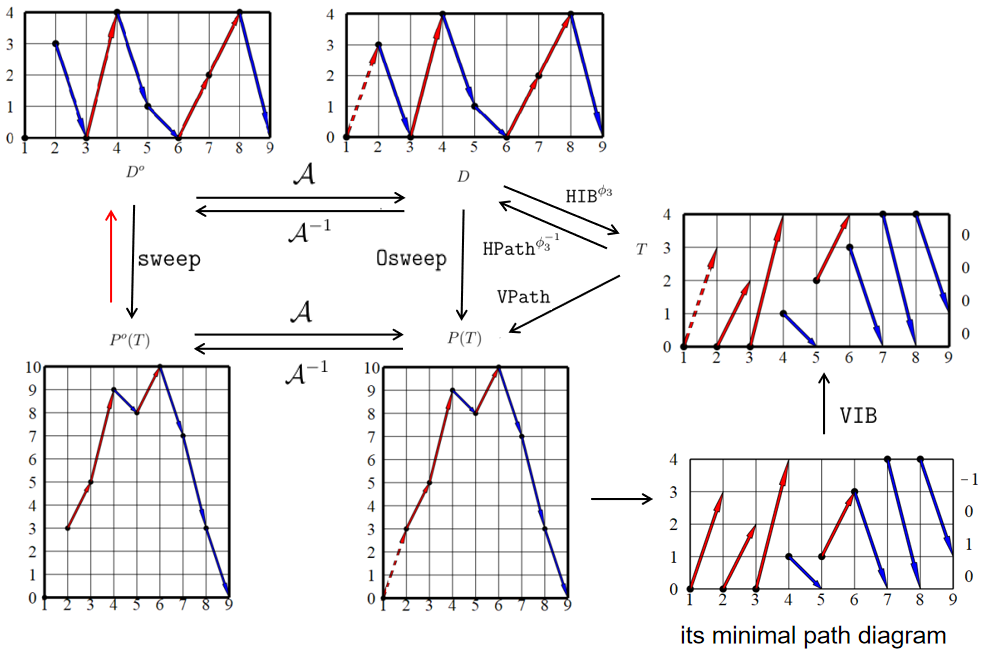}}
\caption{An example of the Sweep map can be obtained by the Order sweep map.}
}
\end{figure}

\begin{theo}
For any multiset $B$ of integers with $|B|<0$ and any sequence of permutations $\phi = (\phi_1, \phi_2, \dots)$, the Order sweep map is a bijection 
$$\osweep : \cal D^o_B \rightarrow \cal D^o_B.$$
\end{theo}

\begin{proof}
Suppose that $\phi_i = \sigma_1\sigma_2\cdots\sigma_i \in \mathfrak{S}_i$.
For every incomplete Dyck path $D^o \in \cal D^o_B$, we can add a red arrow $(1, -|B|)$ to $D^o$ to create a Dyck path $D$, expressed as $D = \mathcal{A}(D^o)$.
The Order sweep map image of $D^o$ with permutations $\phi= (\phi_1, \phi_2, \dots)$ is identical to the Order sweep map image of $D$ with permutations $\phi'= (\phi'_1, \phi'_2, \dots)$, where $\phi_i' = 1\sigma_1'\sigma_2'\cdots\sigma_i'$ with $\sigma_j' = \sigma_j + 1$ for all $j$. It follows since the Order sweep map with permutations $\phi'$ is a bijection on $\cal D'^o_B$, defined as $\cal D'^o_B = \{D = \mathcal{A}(D^o)| D^o \in \cal D^o_B\}$.
\end{proof}

\section{Concluding remark}

In this paper, we delineate a class of Order sweep maps applicable to general paths and rigorously establish that these maps are bijections for general Dyck paths. This classification subsumes the sweep maps for both standard and incomplete general Dyck paths as specific instances. Moreover, we confirm that the Order sweep map maintains its bijection property for incomplete general Dyck paths.

It is acknowledged that the seminal work of Thomas-Williams has already successfully inverted the sweep map for the broader category of non-Dyck paths. In light of this, we will investigate the bijection status of the Order sweep map within this more generalized context in a subsequent study.

\end{document}